\documentclass[12pt]{amsart}
\usepackage{amsaddr}
\usepackage{graphicx} 
\usepackage[margin=1in]{geometry}
\usepackage{amsmath}
\usepackage{amsthm}
\usepackage{natbib}
\usepackage[ruled,lined]{algorithm2e}
\usepackage{xcolor}
\usepackage{booktabs}
\usepackage{array}
\usepackage{multirow}

\theoremstyle{theorem}
\newtheorem{theorem}{Theorem}
\newtheorem{lemma}{Lemma}
\newtheorem{example}{Example}

\theoremstyle{definition}
\newtheorem{definition}{Definition}

\theoremstyle{remark}
\newtheorem*{remark}{Remark}

\DeclareMathOperator{\syl}{Syl}
\DeclareMathOperator{\res}{Res}
\DeclareMathOperator{\sign}{sign}

\SetKw{Continue}{continue}

\title{A new method for reducing algebraic programs to polynomial programs}

\author{Muhammad Maaz}
\address{University of Toronto}
\author{Adam W. Strzebo{\'n}ski}
\address{Wolfram Research}

\date{}

\begin{document}

\maketitle

\begin{abstract}
    We consider a generalization of polynomial programs: algebraic programs, which are optimization or feasibility problems with algebraic objectives or constraints. Algebraic functions are defined as zeros of multivariate polynomials. They are a rich set of functions that includes polynomials themselves, but also ratios and radicals, and finite compositions thereof. When an algebraic program is given in terms of radical expressions, a straightforward way of reformulating into a polynomial program is to introduce a new variable for each distinct radical that appears. Hence, the rich theory and algorithms for polynomial programs, including satisfiability via cylindrical algebraic decomposition, infeasibility certificates via Positivstellensatz theorems, and optimization with sum-of-squares programming directly apply to algebraic programs. We propose a different reformulation, that in many cases introduces significantly fewer new variables, and thus produces polynomial programs that are easier to solve. First, we exhibit an algorithm that finds a defining polynomial of an algebraic function given as a radical expression. As a polynomial does not in general define a unique algebraic function, additional constraints need to be added that isolate the algebraic function from others defined by the same polynomial. Using results from real algebraic geometry, we develop an algorithm that generates polynomial inequalities that isolate an algebraic function. This allows us to reformulate an algebraic program into a polynomial one, by introducing only a single new variable for each algebraic function. On modified versions of classic optimization benchmarks with added algebraic terms, our formulation achieves speedups of up to 50x compared to the straightforward reformulation.
\end{abstract}

\section{Introduction}

We consider programs of the following form:
\begin{align*}
    &\min{f(x)} \qquad \text{s.t.} \\
    &g_1(x) \bowtie_1 0, g_2(x) \bowtie_2 0, g_3(x) \bowtie_3 0, \ldots g_m(x) \bowtie_m 0
\end{align*}
Where $x \in \mathbb{R}^n$, $\{\bowtie_i\}_1^m \in \{=, >, \geq, <, \leq\}^m$, and $f, \{g_i\}_1^m : \mathbb{R}^n \mapsto \mathbb{R}$ are \textit{algebraic functions}. We call such a program an \textit{algebraic program}.

Algebraic functions are a rich class of functions that are defined as the root, i.e., the zero, of a multivariate polynomial with rational coefficients \citep{knopp1996} (we denote the ring of polynomials with rational coefficients $\mathbb{Q}[x]$). For example, $y = \sqrt{x}$ is the root of $y^2-x$. This is contrasted with transcendental functions, like $\sin{x}$. In the case that it is a root of a univariate polynomial, we call it an algebraic \textit{number}, e.g., $\sqrt{2}$ is algebraic as it is a root of $x^2-2$. This is contrasted with transcendental numbers, like $\pi$. Clearly, every polynomial in $\mathbb{Q}[x]$ is algebraic, as the polynomial $f(x)$ satisfies the polynomial equality $y - f(x) = 0$.

Due to the Abel-Ruffini Theorem \citep{ruffini1813, abel1881_1824, abel1881_1826}, not all polynomial roots can be expressed in terms of radicals, that is using only addition, subtraction, multiplication, division, and rational exponents. For example, the algebraic function $f(x)$ implicitly defined by $f(x)^5 - f(x) - 1 = 0$ cannot be expressed with these operations. In this paper, we are only interested in algebraic functions \textit{that can be written out explicitly in terms of radicals}.

This leads us to a constructive perspective on algebraic functions. Any function consisting of a finite number of operations involving only addition, subtraction, multiplication, division, and rational exponents, i.e., a \textit{radical expression}, is algebraic. This is not obvious, and is typically proven by abstract algebra arguments. We refer the reader to the textbook of \citet{vanderwaerden1931} for a full exposition.

\subsection{Related Literature}

In the case that the functions in our program are all polynomials, we call this a polynomial program. This class of problems has been recently extensively studied, with a wealth of theoretical results and algorithms. (If they are all linear, then of course we have a linear program, which perhaps requires no introduction.)

A set defined by polynomial inequalities is called a \textit{basic semialgebraic set}. Its feasibility, or satisfiability in the computer science community, can be determined by means of cylindrical algebraic decomposition (CAD) \citep{collins1974quantifier}. The algorithm of CAD was developed to solve the problem of the existential theory of the reals, which asks for a satisfying assignment to a Boolean combination of polynomial (in)equalities (this defines a \textit{semialgebraic set}). While it was known to be solvable by \citet{tarski1951decision}, Collins' CAD algorithm was the first practical algorithm to solve it.  CAD is now the backbone of satisfiability modulo theories (SMT) solvers like Z3 \citep{z3paper}, which use it as the background theory solver for what is called the theory of non-linear arithmetic, i.e., classic Boolean satisfiability (SAT) generalized to allow for clauses containing polynomial inequalities \citep{z3cad}. See \citet{basu} for more details on algorithms for analyzing polynomial systems.

There is also a rich theory on basic semialgebraic sets and non-negative polynomials originating from the work of Hilbert. Hilbert's Nullstellensatz (``zero locus theorem") is a foundational result in algebraic geometry establishing conditions under which a set of polynomials have no common (complex) zeros \citep{hilbert1893}. There are now other analogous (real) ``positive locus theorems", namely the Positivstellensatz theorems of Krivine-Stengle \citep{krivine1964, stengle1974nullstellensatz}, Schmüdgen \citep{schmudgen1991}, and Putinar \citep{putinar1993}. These theorems provide conditions over which a set of polynomials is non-negative or positive over a semialgebraic set. As a consequence, they provide a certificate of infeasibility of a semialgebraic set, similar to Farkas' lemma for polyhedra.

These theorems are closely related to the concept of sum-of-squares, which goes back to Hilbert's seventeenth problem, which asked if every non-negative polynomial can be written as sum of squares of rational functions, and was answered in the affirmative by \citet{artin1927}. Sum-of-squares was eventually connected to semidefinite programming and birthed the modern field of sum-of-squares programming, a method for polynomial optimization, in the seminal works of \citet{parrilo2003semidefinite} and \citet{lasserre2001global}.

In contrast to polynomials, there has been relatively fewer work on feasibility or optimization with algebraic functions. A straightforward reduction to polynomial problems, which effectively introduces a new variable for each distinct radical, has been implemented in the \emph{Mathematica} system since late 1990s; see \citet{strzebonski1999optim}. \citet{jibetean2006global} study the minimization of a ratio of polynomials (which is of course an algebraic function). There is also the subfield of optimization referred to as geometric programming, which studies programs with functions that are ``posynomials": sums of functions of the form $c x_1^{a_1} x_2^{a_2} \dots x_n^{a_n}$, where the arguments $x_i > 0$, as is the coefficient $c>0$, and $a_i \in \mathbb{R}^n$ \citep{boyd2007tutorial}. These are also algebraic functions, e.g., $2x_1^2x_2^{-1/3}$, and are contrasted from polynomials due to the positive restriction on the arguments and the coefficients, while allowing for any real exponent. These programs can be made convex through a change of variables; see \citet{zener1967geometric}. Some subsequent work has looked at specific problems relating to subclasses of algebraic functions, like \citet{laraki2012semidefinite} and \citet{vigneron2014geometric}. Currently, algebraic functions are likely handled by: ``manual" reformulations, if simple enough, into polynomials, before being input into a polynomial optimization solver; piecewise linear approximations, which can be automatically computed by many solvers like Gurobi; or, generic nonlinear solvers like BARON.

\subsection{Outline of the Idea}

Our approach will be to reformulate algebraic programs into polynomial programs. We begin with a simple motivating example.

\begin{example}
	As a simple example of an algebraic program, we consider the simple problem of $\min \sqrt{x}$, with the only constraint that $x \geq 0$ to keep the problem real-valued. Obviously, the objective is monotonic, and the minimum is attained at $x=0$, but this example will serve as a useful launching point. Suppose we want to reformulate this algebraic program into a polynomial program. We introduce another variable $z$, and fix $z = \sqrt{x}$. Then, we can rewrite our objective as $\min z$. We still have the algebraic constraint that $z = \sqrt{x}$, but we can simply square both sides and obtain the polynomial constraint $z^2 - x = 0$. However, by squaring, we actually have introduced extraneous solutions to the feasible set, as $z = -\sqrt{x}$ also satisfies this constraint. We only want the positive square root. The easy fix is to introduce the constraint $z \geq 0$. Hence, the algebraic program $\min \sqrt{x}$ subject to $x \geq 0$ is equivalent to the polynomial program $\min z$ such that $z^2 - x = 0, x \geq 0, z \geq 0$.
\end{example}

There are effectively two key steps outlined in the example above. First, by finding $z^2-x=0$, we found a \textit{defining polynomial} of $\sqrt{x}$, that is a polynomial of which $\sqrt{x}$ is a root. Secondly, we saw that a defining polynomial may have other roots, and so does not uniquely define the algebraic function. We needed to restrict the possible solutions to the algebraic function we originally had. In complex analysis, this is referred to as isolating the ``branches" of an algebraic function: they would consider $z=\sqrt{x}$ and $z=-\sqrt{x}$ to be the same algebraic function implicitly defined by $z^2-x=0$, but with two ``branches", the positive part and the negative part.

Hence, to reformulate an algebraic program into a polynomial program, we need to be able to, for each algebraic function:

\begin{enumerate}
	\item Find a defining polynomial of the algebraic function
	\item Isolate the algebraic function by introducing additional polynomial constraints
\end{enumerate}

While we saw it was trivial with the $\sqrt{x}$ example, these two steps are not so easy in general. We found a defining polynomial of $\sqrt{x}$ by simple algebra, but algebraic functions can get quite complicated. For example, even with $z = x^{1/2} + x^{1/3}$, it is not obvious how to get rid of the radicals here, as raising both sides to the sixth power introduces additional radicals from the binomial expansion. As well, isolating an algebraic function from the other roots of a defining polynomial via polynomial inequalities is not obvious either. For example, $z = \sqrt{x} - \sqrt{y}$ is not always non-negative so a simple $z \geq 0$ does not suffice.

\subsubsection{Analogies with Algebraic Numbers}

Both of these desired steps have well-established analogies to the case of \textit{algebraic numbers}. Computer algebra systems can represent and manipulate rational numbers exactly. Algebraic numbers are stored internally by their minimal polynomial and a root index, e.g., $\sqrt{2}$ is represented as the second root, according to the usual ordering over the reals, of $x^2-2$. An alternative way to isolate this algebraic number from the other roots of its minimal polynomial is to use an isolating interval defined by rational numbers, which can be computed with classical algorithms like Sturm sequences and Descartes' rule of signs \citep{akritas2005comparative}. For example, $\sqrt{2}$ can be isolated as the sole root of $x^2-2$ in $[1,2]$.

\subsubsection{Our Contributions}

In this paper, we propose a new method of reformulating an algebraic program into a polynomial program. First, we present an algorithm which constructs a defining polynomial for a given radical expression. Then we describe a method that, for several classes of problems, finds polynomial inequalities that isolate the algebraic function represented by the given radical expression from the other roots of its defining polynomial. The proposed method introduces \emph{one} new variable for each distinct radical expression that appears in the algebraic program. A straightforward way of reformulating an algebraic program into a polynomial program introduces one new variable for each distinct radical, hence our reformulation method may introduce significantly less variables. We show examples in which the new approach leads to substantial performance improvements.

\section{Defining Polynomial of a Radical Expression}

As stated in the introduction, we are only interested in algebraic functions that can be written explicitly as radical expressions, that is using only addition, subtraction, multiplication, division, and rational exponents. A \emph{defining polynomial} of a radical expression $f(x_1, \ldots, x_n)$ is a polynomial $p(z, x_1, \ldots, x_n)$ such that $p(f(x_1, \ldots, x_n), x_1, \ldots, x_n)$ is identically zero. The defining polynomial is not unique -- the defining polynomial of lowest degree is called the \textit{minimal polynomial}.

\begin{example}
	\label{ex:noirreducibledefpoly}
	A radical expression may not have an \emph{irreducible} defining polynomial. Let $f(x) = \sqrt{x^{2}}$. Then $p(z, x) = z^2-x^2 = (z-x)(z+x)$ is a defining polynomial of $f(x)$, but neither of its two irreducible factors is identically zero at $f(x)$.
\end{example}

We can understand a radical expression as being made up of ``smaller" radical expressions. Given defining polynomials of two radical expressions, we wish to construct a defining polynomial of some arithmetic combination of these two functions (e.g., their sum). This key insight was used by \citet{zippel1997zero} to extend polynomial zero testing to algebraic function zero testing.

First, we need to define a crucial operator for polynomials, the resultant. Suppose we have two multivariate polynomials $p$ and $q$ with rational coefficients. The resultant is an operator that is applied with respect to one variable, so for its purposes we can treat these polynomials as univariate polynomials in one of the variables. Thus, we fix one of these variables, which we denote $x$, and treat $p$ and $q$ as polynomials in $x$, so that their coefficients are polynomials in the other variables. Let $p$ and $q$ have degrees $d$ and $e$, respectively, so that we can write $p = a_d x^d + a_{d-1} x^{d-1} + \dots + a_0$ and  $q = b_e x^e + b_{e-1} x^{e-1} + \dots + b_0$. The resultant is defined in terms of the Sylvester matrix.

\begin{definition}[Sylvster matrix \citep{basu}]
	The Sylvester matrix of $p$ and $q$, with respect to $x$, is:
	\begin{align*}
	\syl_x(p,q) =
	\begin{bmatrix}
		a_d & a_{d-1} & \cdots & a_0 & 0 & \cdots & 0 \\
		0 & a_d & \cdots & a_1 & a_0 & \cdots & 0 \\
		\vdots & \vdots & \ddots & \vdots & \vdots & \ddots & \vdots \\
		0 & 0 & \cdots & a_d & a_{d-1} & \cdots & a_0 \\
		b_e & b_{e-1} & \cdots & b_0 & 0 & \cdots & 0 \\
		0 & b_e & \cdots & b_1 & b_0 & \cdots & 0 \\
		\vdots & \vdots & \ddots & \vdots & \vdots & \ddots & \vdots \\
		0 & 0 & \cdots & b_e & b_{e-1} & \cdots & b_0
	\end{bmatrix}
	\end{align*}

	Precisely, its rows are: $x^{e-1} p, \, \ldots, \, p, \, x^{d-1} q, \, \ldots, \, q$, considered as vectors in the basis $x^{d + e - 1}, \, \ldots, \, x, \, 1$. Its first $e$ rows are copies of the coefficients of $p$, and its next $d$ rows are copies of the coefficients of $q$. In total, it has $d+e$ rows and $d+e$ columns.
\end{definition}

\begin{definition}[Resultant \citep{basu}]
	The resultant of $p$ and $q$, with respect to $x$, is $\res_x(p,q) = \det \syl_x(p,q)$.
\end{definition}

The resultant is a polynomial in the variables other than $x$, which vanishes if and only if the polynomials $p$ and $q$ have a common root. In this way, it can be used to eliminate a variable from two polynomials, and indeed was originally developed to solve systems of polynomial equalities. This is exactly how resultants will help us construct defining polynomials, as can be seen in the following example.

\begin{example}
	Suppose we have the algebraic function $f(x) = x^{1/2} + x^{1/3}$. It is the sum of two ``smaller" algebraic functions, $x^{1/2}$ and $x^{1/3}$. Let $t_1 = x^{1/2}$ and $t_2 = x^{1/3}$, so that $t_1^2 - x = 0$ and $t_2^3 - x = 0$. Observe that these are indeed defining polynomials of these two algebraic functions. Let $z = f(x) = t_1 + t_2$, so that $z - t_1 - t_2 = 0$. We now have a system of polynomial equalities: $z - t_1 - t_2 = 0, t_1^2 - x = 0, t_2^3 - x = 0$.

	We can then eliminate the two auxiliary variables $t_1$ and $t_2$ by two resultant calculations: namely, $\res_{t_2} (\res_{t_1} (z - t_1 - t_2, t_1^2 - x), t_2^3 - x)$. This gives $-x^3 + 3x^2z^2 - 6x^2z + x^2 - 3xz^4 - 2xz^3 + z^6$, which is indeed a defining polynomial of $x^{1/2} + x^{1/3}$, as can be seen by substituting $z = x^{1/2} + x^{1/3}$.
\end{example}

The example demonstrates that the resultant will be important for determining a defining polynomial of an arbitrary algebraic function by combining defining polynomials of its component functions. More generally, we have the following useful relationships between defining polynomials.

\begin{lemma}[{\citep[Section 9.4]{zippel1993book}}]
	\label{lem:zippel}
	Suppose $f$ and $g$ are algebraic functions, with respective defining polynomials $p(z)$ and $q(z)$, i.e., we fix their main variables to be $z$, and let the degree of $p$ be $d$. Let $r \in \mathbb{Z}, r>0$. Then, the second column gives a defining polynomial of the corresponding entry in the first column:
	\begin{align*}
		f + g &\qquad \res_t \left( p(z - t), q(t) \right) \\
		f - g &\qquad \res_t \left( p(z + t), q(t) \right) \\
		f \times g &\qquad \res_t \left( t^d \cdot p\left(\frac{z}{t}\right), q(t) \right) \\
		\frac{f}{g} &\qquad \res_t \left( p(t z), q(t) \right) \\
		\sqrt[r]{f} &\qquad \res_t \left( z^r - t, p(t) \right),
	\end{align*}

	where $t$ is an auxiliary variable.
\end{lemma}

\begin{remark}
	Note that above, arguments in brackets mean that the function is evaluated at that point, i.e., $q(t)$ means to substitute $t$ in for $z$, as we fixed $z$ to be the main variable.
\end{remark}

This suggests a recursive algorithm for finding a defining polynomial: pattern match the function to one of the cases above, get defining polynomials of the two component functions (recursive step), and return the appropriate resultant. The base case is that if an algebraic function $f$ is just a polynomial, then return the trivial defining polynomial $z - f$. We are now ready to present the algorithm for finding a defining polynomial of an arbitrary radical expression, in Algorithm \ref{algo:finddefpoly}, and we prove its correctness below.

\begin{algorithm}
	\label{algo:finddefpoly}
	\SetAlgoLined
	\KwIn{Radical expression $f(x_1, \ldots, x_n)$.}
	\KwOut{Defining polynomial $p(z, x_1, \ldots, x_n)$ for $f$.}

	\BlankLine
	\SetKwFunction{RecursiveDefPoly}{RecursiveDefPoly}
	\SetKwProg{Fn}{Function}{:}{}

	\Fn{\RecursiveDefPoly{$f$}}
	{
		\uIf{$f$ is a polynomial}{
			\Return $z - f$\;
		}
		\uElseIf{$f = g + h$}{
			$p_g(z) \gets$ \RecursiveDefPoly{$g$}\;
			$p_h(z) \gets$ \RecursiveDefPoly{$h$}\;
			\Return $\res_t \left( p_g(z - t), p_h(t) \right)$\;
		}
		\uElseIf{$f = g - h$}{
			$p_g(z) \gets$ \RecursiveDefPoly{$g$}\;
			$p_h(z) \gets$ \RecursiveDefPoly{$h$}\;
			\Return $\res_t \left( p_g(z + t), p_h(t) \right)$\;
		}
		\uElseIf{$f = g \times h$}{
			$p_g(z) \gets$ \RecursiveDefPoly{$g$}\;
			$p_h(z) \gets$ \RecursiveDefPoly{$h$}\;
			$d \gets$ degree of $p_g(z)$\;
			\Return $\res_t \left( t^d \cdot p_g\left(\frac{z}{t}\right), p_h(t) \right)$\;
		}
		\uElseIf{$f = \frac{g}{h}$}{
			$p_g(z) \gets$ \RecursiveDefPoly{$g$}\;
			$p_h(z) \gets$ \RecursiveDefPoly{$h$}\;
			\Return $\res_t \left( p_g(t z), p_h(t) \right)$\;
		}
		\ElseIf{$f(x) = \sqrt[r]{g}, \quad r \in \mathbb{Z}, r>0$}{
			$p_g(z) \gets$ \RecursiveDefPoly{$g$}\;
			\Return $\res_t \left( z^r - t, p_g(t) \right)$\;
		}
	}

	\caption{Compute a Defining Polynomial of a Radical Expression}
\end{algorithm}

\begin{theorem}
	Algorithm \ref{algo:finddefpoly} computes a defining polynomial of any radical expression.
\end{theorem}
\begin{proof}
	Because every radical expression consists of a finite number of additions, subtractions, multiplications, divisions, and integer radicals, the algorithm terminates in a finite number of steps. The procedure \texttt{RecursiveDefPoly} correctly produces a defining polynomial due to Lemma \ref{lem:zippel}.
\end{proof}

There are some practical implementation considerations to discuss. The polynomial returned by \texttt{RecursiveDefPoly} may not be irreducible. To reduce the degree of the result, the polynomial should be replaced by its square-free part. Another way of attempting to lower the degree is to factor the polynomial and check whether any of the factors can be symbolically simplified to zero after replacing $z$ with $f(x_1, \ldots, x_n)$. Pattern matching for the cases in Algorithm \ref{algo:finddefpoly} has to be handled by the underlying computer algebra system. This is not hard to do in practice as radical expressions are typically stored as an expression tree, and so it suffices to keep checking for the operator at the root of this tree in order to determine the appropriate case.

Another consideration is the handling of rational exponents, e.g., $x^{2/3}$ as Lemma \ref{lem:zippel} only applies to positive integer radicals. But, of course, if $g^{r/s}$ for $r,s \in \mathbb{Z}, r,s>0$, then $g^r$ is algebraic, so we can treat it as $(g^r)^{1/s}$, and then use the result from Lemma \ref{lem:zippel}. Similarly, negative exponents can be treated as division by a positive exponent. Such transformations have to be handled by the computer algebra system but do not affect the pseudocode in Algorithm \ref{algo:finddefpoly}.

\begin{example}
	Some radical expressions (left) and their defining polynomials (right), as determined by Algorithm \ref{algo:finddefpoly}, are:
	\begin{equation*}
	\begin{array}{l l}
		x^{1/2} + x^{1/3} &\quad x^2 - x^3 - 6x^2 z + 3x^2 z^2 - 2x z^3 - 3x z^4 + z^6 \\
		x^{1/2} - x^{1/3} &\quad x^2 - x^3 + 6x^2 z + 3x^2 z^2 + 2x z^3 - 3x z^4 + z^6 \\
		\sqrt{1+x^2} &\quad -1 - x^2 + z^2 \\
		\sqrt{1+x^2} + x/y &\quad x^2 - y^2 - x^2 y^2 - 2 xyz + y^2 z^2
	\end{array}
	\end{equation*}
\end{example}

The complexity of algorithms for polynomial optimization crucially depends on the degrees of the polynomials involved, so we want to bound the degree of the defining polynomial. As the defining polynomial is constructed by iterated resultants, we first prove the following lemma about the degrees of resultants.

\begin{lemma}
	Suppose that we have two $n$-variate polynomials, $p(x_1, \dots, x_n), q(x_1, \dots, x_n)$. Let the degree of $p$ with respect to some $x \in \{x_i\}_1^n$ be denoted $d_{p,x}$, and similarly for $q$ we denote it $d_{q,x}$. Now choose some $x \in \{x_i\}_1^n$, and take the resultant with respect to it: $\res_x(p, q)$. Then, for any $x' \in \{x_i\}_1^n \setminus x$, we have that $d_{\res,x'} = d_{p,x'} d_{q,x} + d_{p,x} d_{q,x'}$.
\end{lemma}
\begin{proof}
	\label{lem:resdegrees}
	Form the Sylvester matrix $\syl_x(p,q)$. The first $d_{q,x}$ rows are coefficients (polynomials in $\{x_i\}_1^n \setminus x$) of $p$, and the next $d_{p,x}$ rows are coefficients (polynomials in $\{x_i\}_1^n \setminus x$) of $q$. Note the structure of the Sylvester matrix is the entries are copied and shifted. Using the Leibniz formula for the determinant, we see that there exists a permutation such that we pick the coefficient from $p$ that achieves degree $d_{p,x'}$ in $x'$, repeated $d_{q,x}$ times (as the entries are just copied and shifted over), and the coefficient from $q$ that achieves degree $d_{q,x'}$ in $x'$, $d_{p,x}$ times. Hence, when we multiply across this permutation, we get a polynomial with degree $d_{p,x'} d_{q,x} + d_{p,x} d_{q,x'}$ in $x'$. Summing across all permutations, we see that this is the degree of the final polynomial.
\end{proof}

Now we can analyze how degrees accumulate during the recursive calls in Algorithm \ref{algo:finddefpoly}.

\begin{theorem}
	Suppose $f(x_1, \dots, x_n)$ and $g(x_1, \dots, x_n)$ are algebraic functions, with respective defining polynomials $p(z, x_1, \dots, x_n)$ and $q(z, x_1, \dots, x_n)$. We denote their degrees with respect to a given variable by $d_{p,\cdot}$ and $d_{q,\cdot}$, respectively, where the placeholder is for the variable. Let $r \in \mathbb{Z}, r>0$. Then, the defining polynomial of the first column, constructed as in Lemma \ref{lem:zippel}, has degree in a variable $x \in \{x_i\}_1^n$ at most the second column, and degree in variable $z$ at most the third column:
	$$
	\begin{array}{l l l}
		f + g &\qquad d_{p,x} d_{q,z} + d_{p,z} d_{q,x} &\qquad d_{p,z} d_{q,z} \\
		f - g &\qquad d_{p,x} d_{q,z} + d_{p,z} d_{q,x} &\qquad d_{p,z} d_{q,z} \\
		f \times g &\qquad d_{p,x} d_{q,z} + d_{p,z} d_{q,x} &\qquad d_{p,z} d_{q,z} \\
		\frac{f}{g} &\qquad d_{p,x} d_{q,z} + d_{p,z} d_{q,x} &\qquad d_{p,z} d_{q,z} \\
		\sqrt[r]{f} &\qquad d_{p,x} &\qquad r d_{p,z}
	\end{array}
	$$
\end{theorem}
\begin{proof}
	We look at Lemma \ref{lem:zippel} for the appropriate resultants for these cases. Recall that the resultants are taken with respect to an auxiliary variable $t$. We apply Lemma \ref{lem:resdegrees} throughout.

	We begin with the statement for $x \in \{x_i\}_1^n$. For the cases $f+g$, $f-g$, $fg$, and $f/g$, the degree of the corresponding resultant is $d_{p,x} d_{q,t} + d_{p,t} d_{q,x}$. For these cases, the degrees with respect to $t$ are the same as the degrees with respect to $z$, so we have $d_{p,x} d_{q,z} + d_{p,z} d_{q,x}$. For $\sqrt[r]{f}$, the corresponding resultant is between: a polynomial with no $x$ terms in it and degrees $r$ and $1$ in $z$ and $t$, respectively, with the polynomial $p$. So, we simply have $d_{p,x}$.

	We now prove the statement for $z$. Similar to before, for the cases $f+g$, $f-g$, $fg$, and $f/g$, the degree of the corresponding resultant is $d_{p,z} d_{q,z}$, because in $q$, we substituted $t$ in for $z$. For $\sqrt[r]{f}$, based on our prior discussion, we have $r d_{p,z}$.
\end{proof}

\begin{remark}
	Observe that the degree of $z$ above only changes, during the running of Algorithm \ref{algo:finddefpoly}, in the $\sqrt[r]{f}$ case. In the base case of Algorithm \ref{algo:finddefpoly}, when we have a polynomial, the degree of $z$ of its minimal polynomial is $1$. If we only had the cases of $f+g$, $f-g$, $fg$, and $f/g$, then the degree of $z$ would always be $1$. It's only due to the $\sqrt[r]{f}$ case that its degree changes. In fact, the recursiveness allows us to see that an upper bound on the degree in $z$ of the minimal polynomial of an algebraic function is at most the product of all such positive integer radicals. This makes intuitive sense as exponents on the $z$ can be thought of as arising from exponentiation to eliminate radicals, e.g., as in $\sqrt{x}$ to $z^2-x$. Another example of this phenomenon is that the minimal polynomial of $x^{1/2} + x^{1/3}$ will have degree in $z$ equal to $6$, due to $2 \times 3$. Or, the minimal polynomial of $\sqrt{\sqrt{x} + 1}$ will have $z$ degree $4$, due to the two square roots.
\end{remark}

\begin{remark}
	The degree bounds above are tight in the case that the \texttt{RecursiveDefPoly} procedure in Algorithm \ref{algo:finddefpoly} actually returns the minimal polynomial. However, it is not easy to check whether the returned polynomial is minimal, as it requires factoring, and even with factoring, the factors themselves may not be defining polynomials of the radical expression, as shown in Example \ref{ex:noirreducibledefpoly}. In practice, we will nonetheless attempt to reduce the degree of the polynomial returned by \texttt{RecursiveDefPoly} by factoring and checking if the radical expression satisfies a factor.
\end{remark}

\section{Isolating an Algebraic Function}

Algorithm \ref{algo:finddefpoly} solves the first part of our plan, to find a defining polynomial. However, we now want to isolate the root representing our radical expression from the other roots of the same defining polynomial by means of polynomial inequalities. To do so, we will use a result from algebraic geometry called Thom's Lemma (see e.g. \citep{basu}), which allows us to distinguish roots of a polynomial by the signs of its derivatives.

\begin{lemma}[Thom's Lemma]
	\label{lemma:thom}
	Let $p(x)$ be a univariate polynomial of degree $d$. Then, for each $(\sigma_0, \ldots, \sigma_d) \in \{-1, 0, 1\}^{d+1}$, the set
		\[
		\left\{ x \in \mathbb{R} \,\middle|\, \sign(p(x))=\sigma_0 \land \sign(p'(x))=\sigma_1 \land \ldots \land \sign(p^{(d)}(x))=\sigma_d \right\}
		\]
		is either empty, a single point, or an open interval.
\end{lemma}

This can be seen with the example of $\sqrt{2}$: it has minimal polynomial $x^2-2$, and at $\sqrt{2}$ the derivative is positive, while at the other root, $-\sqrt{2}$, the derivative is negative. This in fact yields an alternative way to encode algebraic numbers, the so-called Thom encoding \citep{coste1988thom}. We extend this idea of Thom encodings to algebraic functions.

\begin{theorem} \label{thm:rootencoding}
	Let $p(z, x_1, \dots, x_n) = c_d(x_1, \dots, x_n) z^d + c_{d-1}(x_1, \dots, x_n) z^{d-1} + \dots + c_0(x_1, \dots, x_n)$ be a polynomial in $z$, with degree $d$, with coefficients that are polynomials in $x_1, \dots, x_n$. Let $p^{(i)}(z, x_1, \dots x_n)$ denote the $i$-th derivative of $p$ with respect to $z$.

	Define the set $A \subset \mathbb{R}^n$ by:
	\begin{equation*}
	A = \left\{ a \in \mathbb{R}^n \,\middle|\, \res_z \left( p, p^{(i)} \right)(a) \neq 0, \, i = 1, \dots, d \right\}
	\end{equation*}

	Then $A$ is an open subset of $\mathbb{R}^n$ such that for each connected component $C$ of $A$:
	\begin{enumerate}
		\item \label{item:roots} There exist continuous functions $r_1(x_1, \dots, x_n) < \dots < r_k(x_1, \dots, x_n)$ on $C$ such that, for each $a \in C$, the real roots of $p(z, a)$ in $z$ are $\{ r_1(a), \dots, r_k(a) \}$.
		\item \label{item:signs} For each $i = 1, \dots, d$ and $j = 1, \dots, k$, $\sign(p^{(i)}(r_j(x_1, \dots, x_k)))$ is constant and nonzero on $C$.
		\item \label{item:isolation} Let $\sigma_{i, j} := \sign(p^{(i)}(r_j(x_1, \dots, x_k)))$. Then $z = r_j(x_1, \dots, x_n)$ iff $p(z, x_1, \dots, x_n) = 0 \land \sign(p'(z, x_1, \dots, x_n))=\sigma_{1, j} \land \ldots \land \sign(p^{(d)}(z, x_1, \dots, x_n))=\sigma_{d, j}$.
\end{enumerate}
\end{theorem}

\begin{proof}
	Each $C$ is a connected open set on which the discriminant and the leading coefficient of $p$ with respect to $z$ have constant nonzero signs, because $\res_z \left( p, p' \right) \neq 0$ and $\res_z \left( p, p^{(d)} \right) = (d!)^d c_d \neq 0$. Hence, (\ref{item:roots}) follows from Lemma 3.6 of \citep{strzebonski2000gencad}.
    The conditions that define $A$ guarantee that $p$ and its derivatives have no common roots in $A$, which proves (\ref{item:signs}).
    Finally, (\ref{item:isolation}) follows from Thom's Lemma.
\end{proof}

Since radical expressions define continuous functions, Theorem \ref{thm:rootencoding} shows that, on the set $A$, we can characterize the algebraic function given by a radical expression, uniquely, from the signs of the derivatives of its defining polynomial. The sign conditions will be polynomial inequalities, as the derivatives are polynomials. We can now state Algorithm \ref{algo:isolate} which generates a set of polynomial inequalities to isolate a given algebraic function from the other roots of its defining polynomial, provided we can effectively find the connected components of the set $A$ defined in the algorithm.

\begin{algorithm}
	\caption{Isolating an Algebraic Function From Other Roots of its Defining Polynomial}
	\label{algo:isolate}
	\KwIn{A radical expression $f(x_1, \dots x_n): \mathbb{R}^n \mapsto \mathbb{R}$ with defining polynomial $p(z, x_1, \dots, x_n): \mathbb{R}^{n+1} \mapsto \mathbb{R}$ of degree $d$.}
	\KwOut{A list of polynomial inequalities that uniquely define the algebraic function given by $f$ on the set $A = \left\{ x \in \mathbb{R}^n \,\middle|\, \res_z \left( p, p^{(i)} \right) \neq 0, \, i = 1, \dots, d \right\}$.}

	\BlankLine

	$L \gets$ empty list\;

	\ForEach{connected component $C$ of $A$}{
		Choose a point $a \in C$\;
		\uIf{all radicals in $f(a)$ are real-valued}{
			$r \gets f(a)$\;
		}
		\Else{
			\Continue
		}

		$\ell \gets$ conditions defining $C$\;
        Append $p = 0$ to $\ell$\;
		\For{$i = 1$ \KwTo $d$}{
			Compute the sign $s_i = \sign\left( p^{(i)}(r, a) \right)$\;
			\uIf{$s_i = +1$}{Append $p^{(i)} > 0$ to $\ell$\;}
			\Else{Append $p^{(i)} < 0$ to $\ell$\;}
		}
		Append $\ell$ to $L$
	}

	\KwRet{$L$}
\end{algorithm}

The key in Algorithm \ref{algo:isolate} is that over each connected component of $A$, we can determine the necessary sign conditions by evaluation at a single point.

In principle, polynomial inequalities defining connected components of $A$ can always be computed \citep{canny1988components, BPR1999components, strzebonski2017components}. However, generally computing connected components may be costly and the polynomial descriptions of components may be very complicated. Algorithm \ref{algo:isolate} remains correct if instead of connected components of $A$ one uses any decomposition of $A$ into connected sets. Hence, instead of computing connected components one can use cells of a cylindrical algebraic decomposition, or one can only combine cells that can be easily shown to be adjacent. In some instances, additional problem structure may be present that makes finding the components easy. One such case is when the radical expression depends only on one of the variables. Another is when the optimization constraints imply that the resultants in the definition of $A$ are nonzero and the set described by the constraints is obviously connected (e.g. the constraints are linear) -- in this case the set described by the constraints is contained in one connected component of $A$ and we are not interested in other components.

We now demonstrate Algorithm \ref{algo:isolate} with our old example of $\sqrt{x}$.

\begin{example}
	For $\sqrt{x}$, we have the defining (indeed, minimal) polynomial $p = z^2 - x$. The first derivative is $p^{(1)} = 2z$ and the second derivative is $p^{(2)} = 2$. We have $\res_z(p, p^{(1)}) = -4 x,$, so $A = \{x : x \neq 0\}$. The connected components of $A$ are $\{x : x > 0\}$ and $\{x : x < 0\}$.

	Over the component $\{x : x < 0\}$, choose any $x$, say $x = -1$. Plug it into $f$: $f(-1) = \sqrt{-1}$. This is not a real number, so no (real) algebraic functions exist here, which makes sense as the argument of a square root should be non-negative over the reals. So, we skip this component.

	Over the component $\{x : x > 0\}$, choose any $x$, say $x = 1$. Plug it into $f$: $f(1) = \sqrt{1} = 1$. This is a real number, so we have an algebraic function here. We let $r = f(1) = 1$. Evaluate the first derivative of $p$ at $r$ and $a$: $p^{(1)}(1, 1) = 2$. This is positive, so we have the inequality $2z > 0$, or $z > 0$. For the second derivative, we have $p^{(2)}(1, 1) = 2$, which is positive, so we have the inequality $2 > 0$, which is trivial so we do not need to include it.

	Therefore, the inequalities to isolate $\sqrt{x}$ over the component $\{x : x > 0\}$ are $z^2 - x = 0, z > 0$, which corroborates what we would have derived by hand.
\end{example}

We now show a harder example. Observe that it is possible for the same algebraic function to exist in multiple connected components. This is why in Algorithm \ref{algo:isolate}, we return a list of lists: each element of the list corresponds to a connected component containing the algebraic function, and within that is the list of inequalities defining the algebraic function in that component.

\begin{example}
	Consider the polynomial $x^2 - 2xy + y^2 - 2xz^2 - 2yz^2 + z^4$ which implicitly defines four algebraic functions: $\pm \sqrt{x} \pm \sqrt{y}$ (see Figure \ref{fig:fouralgfuncs}). We define $A = \{x \neq 0, y \neq 0, x-y \neq 0, x^2-7xy+y^2 \neq 0\}$. Suppose we want to isolate $\sqrt{x}+\sqrt{y}$. Running through Algorithm \ref{algo:isolate}, we obtain the following conditions to isolate $\sqrt{x}+\sqrt{y}$:
	\begin{align*}
		\{ x > 0, \; y > \frac{1}{2} (7 + 3 \sqrt{5}) x, \; z^3 > (x + y) z, \; x + y < 3 z^2, \; z > 0 \}, \\
		\{ x > 0, \; 0 < y < \frac{1}{2} (7 - 3 \sqrt{5}) x, \; z^3 > (x + y) z, \; x + y < 3 z^2, \; z > 0 \}, \\
		\{ x > 0, \; x < y < \frac{1}{2} (7 + 3 \sqrt{5}) x, \; z^3 > (x + y) z, \; x + y < 3 z^2, \; z > 0 \}, \\
		\{ x > 0, \; \frac{1}{2} (7 - 3 \sqrt{5}) x < y < x, \; z^3 > (x + y) z, \; x + y < 3 z^2, \; z > 0 \}
	\end{align*}

	Each list bracketed by curly braces corresponds to a connected component. In each, the first two conditions define the component and the latter three are the derivative sign conditions. We can see the function $\sqrt{x} + \sqrt{y}$ actually spans four connected components (see Figure \ref{fig:fourcomponents}). However, we can also note that the conditions on $y$ over the four components could be relaxed and combined so that we obtain a single component $x>0,y>0$, with the simpler representation $z^3 > (x+y)z, x+y<3z^2, z>0$ (if we are willing to accept where the resultants are zero).
\end{example}

\begin{figure}
	\includegraphics[width=0.6\textwidth]{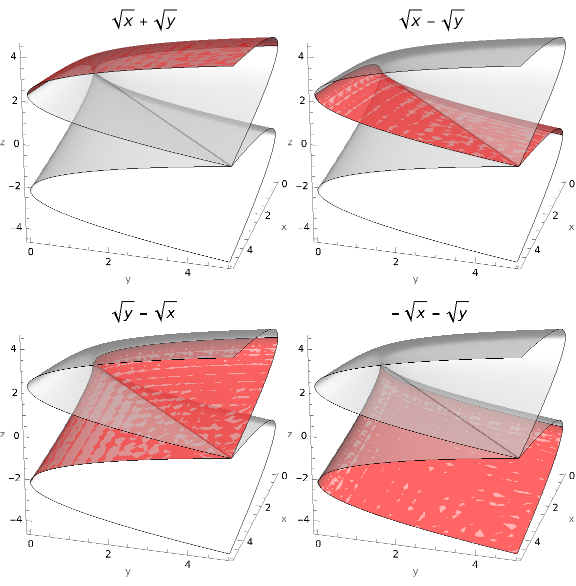}
	\caption{The polynomial $x^2 - 2xy + y^2 - 2xz^2 - 2yz^2 + z^4$ implicitly defines four algebraic functions as its roots: $\pm \sqrt{x} \pm \sqrt{y}$.}
	\label{fig:fouralgfuncs}
\end{figure}

\begin{figure}
	\includegraphics[width=0.6\textwidth]{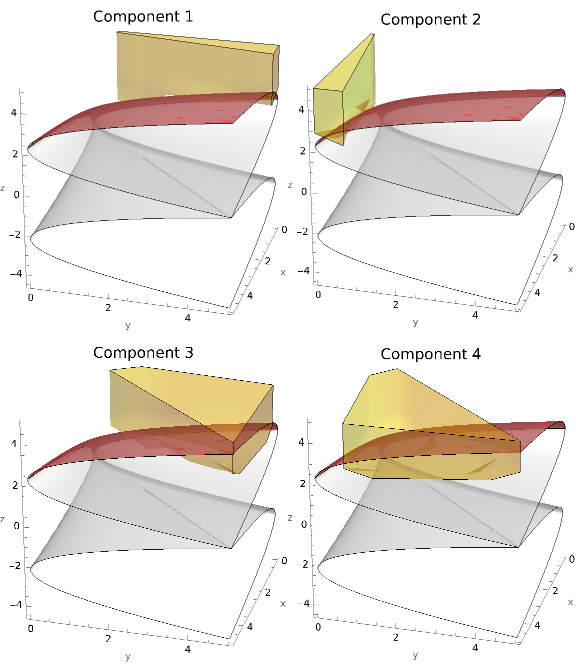}
	\caption{The algebraic function $\sqrt{x}+\sqrt{y}$ (red) spans four connected components on the zero locus of its minimal polynomial $x^2 - 2xy + y^2 - 2xz^2 - 2yz^2 + z^4$ (gray), each defined by polynomial inequalities (yellow).}
	\label{fig:fourcomponents}
\end{figure}

We now bound the complexity of this procedure by bounding the following quantities: the number of connected components over which an algebraic function can exist, the number of polynomial inequalities to isolate an algebraic function, and their degrees.

\begin{theorem}
	Suppose we have an algebraic function $f(x_1, \dots x_n)$, with a defining polynomial $p(z, x_1, \dots, x_n)$, which has degree in $z$ of $d$, and the total degree of each monomial in the resultants that define $A$ is bounded by $e$. Then, the algebraic function $f$ exists over at most $\binom{d}{n} \left( O(e) \right)^n$ connected components. As well, within a connected component, $f$ is uniquely defined by at most $d$ polynomial inequalities, which have degree in $z$ at most $d-1$.
\end{theorem}

\begin{proof}
	The set $A$ in Algorithm \ref{algo:isolate} is defined by $d$ resultants. Also, the algebraic function $f$ lies on an algebraic variety of dimension $n$, i.e., the zero locus of $p$. By \citet{basu1996number}, the number of cells is $\binom{d}{n} \left( O(e) \right)^n$.

	To define $f$ within a connected component, we use the derivative sequence, of which there are $d$ polynomials, which have degrees in $z$ of $d-1, \dots, 0$.
\end{proof}

\section{Reformulating an Algebraic Program}

We now have the necessary tools to reformulate an algebraic program into a polynomial one. Simply, for each radical expression, we run Algorithms \ref{algo:finddefpoly} and \ref{algo:isolate}. Let $A_{all}$ be the intersection of the sets $A$ that appear in all calls to Algorithm \ref{algo:isolate} and let $D$ be the set described by the optimization constraints. If $A_{all}\cap D$ is dense in $D$, the infimum of values of the objective function over $A_{all}\cap D$ and over $D$ are the same, hence it is sufficient to solve the optimization problem over $A_{all}\cap D$. If none of the resultants that appear in the definition of $A_{all}$ is identically zero, $A_{all}$ is open and dense in $\mathbb{R}^n$. If the problem is unconstrained, or all constraints are strict inequalities, or, more generally, if $D$ is contained in the closure of its interior, it follows that $A_{all}\cap D$ is dense in $D$. This are the cases in which we will apply our algorithm. In a more general setting, one would need to check whether $A_{all}\cap D$ is dense in $D$, which in principle can be done using the cylindrical algebraic decomposition algorithm.

We may obtain multiple child problems, due to the presence of different connected components. However, we argue that this is often better than the straightforward approach, which is to introduce a new variable for each distinct radical, as long as the number of distinct radical expressions we need to replace is less than the number of distinct radicals. For example, for $\sqrt{x} + \sqrt{y}, x>0, y>0$, replacing distinct radicals with new variables yields $z_1 + z_2, z_1^2 - x = 0, z_2^2 - y = 0, z_1 > 0, z_2 > 0, x>0, y>0$. Clearly, the number of variables can get quite large if the radical expression gets more complicated. This is important because current algorithms for polynomial programming all scale quite poorly with the number of variables. The complexity of cylindrical algebraic decomposition is doubly exponential in the number of variables \citep{basu}, but polynomial in the degree and number of polynomials. In fact, every algorithm listed on \citet[pages 4-6]{basu} has a complexity that is exponential in the number of variables but polynomial in the degree, including computing a point in each semialgebraic connected component as well as solving problems related to robotics motion planning. The complexity of sum-of-squares programming is similar: for the standard sum-of-squares formulation, the size of the matrices in our semidefinite program is proportional to the number of polynomials of total degree $d$ in $n$ variables, which grows like the binomial coefficient $\binom{n+d}{d}$ \citep{blekherman2012semidefinite}.

It is difficult to come up with a precise quantitative comparison between our method and the straightforward reformulation described above, as it is difficult to measure how complicated an algebraic function is (one measure may be how many recursive calls are made in Algorithm \ref{algo:finddefpoly}). Hence, below we show a few examples to show the utility of our approach.

\subsection{Implementation}

We implemented Algorithms \ref{algo:finddefpoly} and \ref{algo:isolate} in \emph{Mathematica}. Specifically, our implementation finds a defining polynomial, attempts to reduce its degree, and then isolates the root.

For the optimization problem, we use the sum-of-squares relaxation into a semidefinite program \citep{parrilo2003semidefinite, blekherman2012semidefinite} via \texttt{SumOfSquares.py} \citep{pythonsos} and PICOS \citep{picos}, with Mosek as the backend semidefinite solver. We will compare our performance with the straightforward reformulation, which is to introduce a new variable for each distinct radical.

It is important to note that various \emph{numerical optimization} methods are of course applicable to algebraic programs, e.g., gradient descent with random restarts. These can often be quite fast. However, here, we will use the sum-of-squares relaxation of polynomial programs \citep{parrilo2003semidefinite}, which provides formal guarantees on the optimum, which is important in certain science and engineering applications. In this paper, we do not compare our method to numerical optimization methods as our main goal is to extend methods for polynomial programs to algebraic programs.

\subsection{Example: Modified Goldstein-Price Function}

We analyze a modification of the Goldstein-Price function, a classic test case for global (polynomial) optimization \citep{testprob}. The Goldstein-Price function $f_{GP}(x,y)$ is defined as:

\begin{align*}
	f_{GP} = & \left[1 + (x + y + 1)^2 \left(19 - 14x + 3x^2 - 14y + 6xy + 3y^2\right)\right] \\
	& \times \left[30 + (2x - 3y)^2 \left(18 - 32x + 12x^2 + 48y - 36xy + 27y^2\right)\right]
\end{align*}

We translate the Goldstein-Price function down by $\sqrt{x-1} + \sqrt{y-1}$. In other words, we wish to minimize $f_{GP}(x,y) - \sqrt{x-1} - \sqrt{y-1}$, such that $x-1 \geq 0, y-1 \geq 0$.

The straightforward reformulation is: $\min f_{GP}(x,y) - u - v$ with the following constraints:
\begin{align*}
	&u^2 - x+1 = 0, v^2 - y+1 = 0 &\quad \text{(New vars.)} \\
	&x -1 \geq 0, y - 1 \geq 0, u \geq 0, v \geq 0 &\quad \text{(Domain)}
\end{align*}

The reformulation from our approach is: $\min f_{GP} - z$ with the following constraints:
\begin{align*}
	&x^2 - 2xy + y^2 + 4z^2 - 2xz^2 - 2yz^2 + z^4 = 0 &\quad \text{(Def. poly.)} \\
	&z^2 + 2 - x - y \geq 0, z \geq 0 &\quad \text{(Isolating root)} \\
	&x -1 \geq 0, y - 1 \geq 0 &\quad \text{(Domain)}
\end{align*}

We generated a single inequality to isolate the root as we merged the four connected components, hence deriving the constraints $z(2+z^2) \geq (x+y)z, 2+3z^2 \geq x+y, z \geq 0$, and note that the second constraint can be eliminated as we cancel the $z$ from the first constraint and observe that the first constraint includes the second one. Generating the minimal polynomial and the isolating inequalities took 0.135 seconds -- this can be considered the overhead of our formulation.

When constructing the semidefinite relaxation, we have to specify a degree bound to indicate the level of the sum-of-squares hierarchy to use \citep{parrilo2003semidefinite}. As the degree bound goes up, the relaxation is better but comes at a higher computational cost. We run the two formulations above with different degree bounds. The results are summarized in Table \ref{tab:main}.

\begin{table}[h!]
	\centering
	\tiny
	\begin{tabular}{llcccc}
		\toprule
		\textbf{Formulation} & \textbf{Degree} & \textbf{Number of variables} & \textbf{Total time (s)} & \textbf{Semidefinite solver time (s)} & \textbf{Lower bound} \\
		\midrule
		\multirow{3}{*}{Straightforward}
		& 4 & 6266 & 4.77 & 0.30 & 819.11 \\
		& 6 & 60166 & 93.13 & 4.90 & 830.53 \\
		& 8 & 359581 & 1663.12 & 67.70 & 985.33 \\
		\midrule
		\multirow{3}{*}{Ours}
		& 4 & 1842 & 0.99 & 0.06 & 819.11 \\
		& 6 & 11264 & 6.36 & 0.52 & 825.59 \\
		& 8 & 45911 & 39.12 & 2.65 & 985.32 \\
		\bottomrule
	\end{tabular}
	\caption{Performance metrics for different formulations of our modified Goldstein-Price function optimization and different degrees of sum-of-squares hierarchy}
	\label{tab:main}
\end{table}

The differences in the formulations is clear: ours has between 6-8x fewer variables in the semidefinite relaxation, and achieves up to a 40x speedup in total time. The optimal value returned by the semidefinite solver is a lower bound on the objective value of our program, and there are no discernible differences between our formulation and the straightforward one there, showing that the relaxation from our formulation is not weaker.

\subsection{Example: Modified Rosenbrock Function}

Consider the Rosenbrock function: $f_R(x,y) = (1-x)^2 + 100(y-x^2)^2$, which is a classic global optimization test case characterized by a long, narrow valley \citep{testprob}. We add to it a nested radical expression, so our problem is to minimize $f_R(x,y) + \sqrt{x^2 + \sqrt{y^2+1}}$ over $\mathbb{R}^2$.

The straightforward reformulation is $\min f_R(x,y) + u$ with the following constraints:
\begin{align*}
	&u^2 = x^2 + v, v^2 = y^2 + 1 &\quad \text{(New vars.)} \\
	&u \geq 0, v \geq 0 &\quad \text{(Domain)}
\end{align*}

Note that we need not restrict the domain of the argument in the radical as it is always non-negative.

The reformulation from our approach is $\min f_R + z$ with the following constraints:
\begin{align*}
	&z^4 - 2x^2z^2 - y^2 + x^4 - 1 = 0 &\quad \text{(Def. poly.)} \\
	&z^2 \geq x^2, z \geq 0 &\quad \text{(Isolating root)}
\end{align*}

Generating the defining polynomial and the isolating inequalities (to note that they are not needed) took 0.52 seconds. As before, we now construct the semidefinite relaxations for different degrees of the sum-of-squares hierarchy and compare the results.

\begin{table}[h!]
	\centering
	\tiny
	\begin{tabular}{llcccc}
		\toprule
		\textbf{Formulation} & \textbf{Degree} & \textbf{Number of variables} & \textbf{Total time (s)} & \textbf{Semidefinite solver time (s)} & \textbf{Lower bound} \\
		\midrule
		\multirow{3}{*}{Straightforward}
		& 4 & 4586 & 2.99 & 0.19 & 1.35029 \\
		& 6 & 42162 & 58.20 & 2.73 & 1.35029 \\
		& 8 & 244231 & 1111.23 & 34.29 & 1.35031 \\
		\midrule
		\multirow{3}{*}{Ours}
		& 4 & 1254 & 0.72 & 0.05 & 1.35029 \\
		& 6 & 7500 & 3.41 & 0.27 & 1.35029 \\
		& 8 & 30031 & 21.86 & 1.23 & 1.35029 \\
		\bottomrule
	\end{tabular}
	\caption{Performance metrics for different formulations of our modified Rosenbrock function optimization and different degrees of sum-of-squares hierarchy}
	\label{tab:rosenbrock}
\end{table}

Again, we see that our formulation solves the optimization problem much faster than the straightforward one. Even at the lowest degree of the sum-of-squares hierarchy, degree $4$, our formulation solves the problem around 4x faster than the straightforwrad one, and this difference grows to around 50x at the highest degree.

\section{Discussion and Conclusion}

In this paper, we presented a strict generalization of polynomial programming termed algebraic programming. We presented a systematic way to formulate algebraic programs into polynomial programs, by means of a minimal polynomial and root isolation inequalities. The latter algorithm also provides an algorithmic approach to the problem of branch cut generation in complex analysis. On two toy examples with a modified Goldstein-Price function and a modified Rosenbrock function, we demonstrated the utility of our formulation over the straightforward formulation.

There are a few barriers to scalability of our algorithm. In Algorithm \ref{algo:isolate}, the degree of the minimal polynomial can quickly blow up depending on how complex the algebraic function is (though we provided bounds on these degrees), as can the number of monomials. The complexity of finding connected components of the set $A$ in Algorithm \ref{algo:isolate} can also be high as the degree of the minimal polynomial goes up, and it requires the use of cylindrical algebraic decomposition to characterize the connected components. As well, with multiple algebraic functions in the same program, we may get a combinatorial explosion in the number of connected components that we need to search over. In fact, the complexity of Algorithms \ref{algo:finddefpoly} and \ref{algo:isolate} might be prohibitive for some algebraic programs, and in those cases, it is likely easier to just use the straightforward reformulation. However, these complexities match typical complexities seen in algorithms for polynomial systems. Our method may not be applicable in some algebraic programs (when $A_{all}\cap D$ is not dense in $D$), or it may be computationally prohibitive for some instances, and hence cannot be considered a practical panacea for all algebraic programs. As well, for certain classes of algebraic programs, specialized algorithms are likely to be more efficient, e.g., rational functions and posynomials (i.e., geometric programming).

However, when the overhead of Algorithms \ref{algo:finddefpoly} and \ref{algo:isolate} is low enough, as seen on our toy examples of the modified Goldstein-Price function and the modified Rosenbrock function, we gain a considerable speedup in solution time compared to the straightforward reformulation -- in fact, up to a 50x speedup in our experiments. As well, our theoretical contribution of uniquely identifying an algebraic function by a defining polynomial and polynomial inequalities is a new result, which may be of independent interest. In conclusion, we hope that this work will spur further research into algebraic programming and its applications.

\newpage

\bibliographystyle{plainnat}
\bibliography{ref.bib}

\end{document}